\documentclass[11pt,a4paper]{article}

\usepackage[T1]{fontenc}
\usepackage{amsmath, amsfonts, amssymb}

\usepackage{array}
\usepackage{afterpage}

\usepackage{enumitem}

\usepackage{geometry}
\usepackage{graphicx}
\usepackage{imakeidx}

\usepackage{hyperref}

\usepackage{stackengine, stackrel}
\usepackage{setspace}

\usepackage{tabularx}
\usepackage{amsthm, thmtools}
\usepackage{tikz}
\usepackage{titlesec,titletoc}

\usepackage{xcolor}

\usetikzlibrary{cd, backgrounds, calc, shadows, positioning, fit, intersections, through, arrows.meta, bending, fadings, petri}

\renewcommand{\d}{\mathrm{d}}

\newcommand{\loc}{\mathrm{loc}}

\newcommand{\Div}{\operatorname{div}}

\newcommand{\R}{\mathbb{R}}

\newcommand{\eps}{\varepsilon}

\newcommand{\fai}{\varphi}

\newcommand{\br}[1]{\left( #1 \right)}

\renewcommand{\emptyset}{\varnothing}

\hypersetup{
colorlinks=true,
linkcolor=red,
filecolor=blue,
urlcolor=red,
citecolor=green,
}

\allowdisplaybreaks[1]

\setlist{itemsep=0pt,parsep=0pt,partopsep=0pt,topsep=0pt}
\setlist[enumerate,1]{label=(\alph*)}
\setlist[enumerate,2]{label=(\roman*)}
\setlist[enumerate,3]{label=(\arab*)}

\declaretheoremstyle[
headfont=\bfseries,
notefont=\bfseries,
bodyfont=\itshape,
spaceabove=1em,
spacebelow=1em,
notebraces={(}{)},
postheadspace=1ex,
headpunct={.},
headindent=0pt
]{thm}
\declaretheorem[title=Theorem, style=thm, numberwithin=section]{theorem}
\declaretheorem[title=Lemma, style=thm, numberwithin=section]{lemma}

\declaretheorem[sibling=theorem, title=Corollary, style=thm]{corollary}
\declaretheorem[title=Remark, style=thm, numberwithin=section]{remark}

\declaretheoremstyle[
headfont=\itshape,
notefont=\itshape,
bodyfont=\normalfont,
spaceabove=1em,
spacebelow=1em,
notebraces={(}{)},
postheadspace=1ex,
headpunct={.},
headindent=0pt
]{rmk}

\declaretheoremstyle[
headfont=\bfseries,
notefont=\bfseries,
bodyfont=\normalfont,
spaceabove=1em,
spacebelow=1em,
notebraces={(}{)},
postheadspace=1ex,
headpunct={.},
headindent=0em
]{exam}

\counterwithin{equation}{section}
\newcommand\zhmgs[1]{{\bfseries #1} }
\makeatletter
\def\zhm{\@ifstar{\zhmm[Proof]}{\@ifnextchar[{\zhmm}{\zhmm[Proof.]}}}
\def\zhmm[#1]{\@ifnextchar[{\zhmmm[#1]}{\zhmgs{#1}}}
\def\zhmmm[#1][#2]{\zhmgs{#1\ (#2)}}
\makeatother
\renewenvironment{proof}
{\addvspace{1em} \noindent\zhm}
{\hfill $\square$ \par \ignorespacesafterend \addvspace{1em}}

\allowdisplaybreaks

\begin{document}

\title{\bfseries Estimate for the first Dirichlet eigenvalue of $p-$Laplacian on non-compact manifolds}
\author{Xiaoshang Jin\thanks{The 1st-named author was supported by {NNSF of China
			\# 12201225}},
        Zhiwei L\"u}

\date{}

	\maketitle

    \begin{abstract}
        In this paper, we establish a sharp lower bound for the first Dirichlet eigenvalue of the $p$-Laplacian on bounded domains of a complete, non-compact Riemannian manifold with non-negative Ricci curvature.
    \end{abstract}

    \section{Introduction}
    Assume that $(M,g)$ is a complete Riemannian manifold and $p>1$.  Let us use $\Delta_p$ to denote the $p$-Laplace operator, i.e.
    \[
    \Delta_p u={\rm div}(|\nabla u|^{p-2}\nabla u),\ \ \ u\in W^{1,p}_{\rm loc}(M).
    \]
    Here the equality is in the weak $W^{1,p}_{\rm loc}$ sense:
    $$
    \Delta_p u=f\ \Leftrightarrow\ \forall \varphi\in C^\infty_0(M),\ -\int_M |\nabla u|^{p-2}g(\nabla\varphi,\nabla u)\,\d\upsilon_g=\int_M \varphi f\,\d\upsilon_g
    $$ If there exists a constant $\lambda$ and a function $u\neq 0$ such that
    \[
    \Delta_p u=-\lambda|u|^{p-2} u,
    \]
    then we say that $\lambda$ is a eigenvalue of $\Delta_p$ with eigenfunction $u.$ There are lots of researches concerning the first $p-$eigenvalue of closed manifolds and compact manifolds with Neumann condition if $\partial M\neq \emptyset,$ especially  under assumption of a lower Ricci curvature bound. (See e.g.,     \cite{matei2000first}\cite{kawai2003first}\cite{zhang2007lower}\cite{valtorta2012sharp}\cite{naber2014sharp} ).
    \par In this paper, we are interested in the Dirichlet problem. More concretely, let $\Omega$ be a bounded domain of $(M,g)$ with piecewise smooth boundary $\partial\Omega\neq\emptyset.$ We denote the first positive eigenvalue as $\lambda_{1,p}(\Omega)$ with the Dirichlet condition:
    \[\begin{cases}
     \Delta_p u=-\lambda_{1,p}(\Omega)|u|^{p-2} u\ & {\rm in}\ \Omega;\\
     u=0 & {\rm on}\ \partial\Omega.
     \end{cases}
    \]
    This eigenvalue has the following variational characterization:
    \begin{equation}\label{eq:1.1}
        \lambda_{1,p}(\Omega) = \inf_{v \in W^{1,p}_{0}(\Omega) \setminus \{0\}} \frac{\int_\Omega |\nabla v|^p \,\d\upsilon_g}{\int_\Omega |v|^p \,\d\upsilon_g}.
    \end{equation}
    The main result of this paper is a sharp lower bound for $\lambda_{1,p}(\Omega)$ under the condition that $(M,g)$ is a non-compact and  non-negative Ricci curved Riemannian manifold:
    \begin{theorem}\label{thm:1.1}
      Let $(M,g)$ be a complete non-compact manifold with ${\rm Ric}\geq 0.$ Then for any bounded domain $\Omega\subseteq M$ and $p>1,$
        \begin{equation}\label{eq:1.2}
            \lambda_{1,p}(\Omega) > (p-1)\br{\frac{\pi_p}{2d}}^p
        \end{equation}
        where $d={\rm diam}(\Omega)$ is the diameter of $\Omega$ and
        \[
            \pi_p := 2\int_0^1 \frac{1}{(1-t^p)^{\frac 1p}} \,\d t = \frac{2\pi}{p\sin \frac{\pi}{p}}.
        \]
        Moreover, the estimate is sharp.
    \end{theorem}
    \begin{remark}
       The estimate of \eqref{eq:1.2} is sharp in the following sense: there exists a sequence of non-compact manifolds $(M_i,g_i)$ and bounded domains $\Omega_i\subseteq M_i$ with ${\rm diam}(\Omega_i)=d$  such that
    $$\lim\limits_{i\rightarrow\infty}\lambda_{1,p}(\Omega_i)= (p-1)\Big(\frac{\pi_p}{2d}\Big)^p
    $$
    \end{remark}
    \begin{corollary}
     Under the conditions of Theorem \ref{thm:1.1}, if $\Omega=B(x,1)$ is a geodesic ball of radius 1 centred at $x,$ then
     $d\leq 2$ and
        \begin{equation}\label{eq:1.3}
            \lambda_{1,p}(B(x,1)) > (p-1)\Big(\frac{\pi_p}{4}\Big)^p
        \end{equation}
        and the estimate is sharp.
    \end{corollary}
    \begin{remark}
      The condition in Theorem \ref{thm:1.1} that $M$ is non-compact cannot be dropped, since we have a sequence of bounded domain $M_i$ embedded in the standard sphere $\mathbb{S}^n$ satisfying  that $\lambda_{1,p}(M_i) \to 0$ when $p\leq n.$ See Remark 2.1. for more details. As a consequence, we can obtain the following property:
     \end{remark}
      \begin{corollary} \label{coro:1.3}
      For any $n\geq 2,$ there exists a compact domain $D$ strictly contained in an $n-$dimensional manifold $M$ satisfying that $Ric_M\geq 0,$ such that for any non-compact complete manifold $M'$ of dimension $n$ with non-negative Ricci curvature, $D$ is not isometrically embedded in $M'.$
    \end{corollary}
    Let us recall that in \cite{yang1999lower}, the author proved that $ \lambda_{1,2}(\Omega) \geq\frac{\pi^2}{4r^2}$ for compact manifold $\Omega$ with mean convex boundary $\partial \Omega$ and non-negative Ricci curvature where $r$ is the inscribed radius for $\Omega.$ This result was later generalized to $p-$Laplasian in \cite{zhang2007lower}:
    $$
    \lambda_{1,p}(\Omega) \geq (p-1)\br{\frac{\pi_p}{2r}}^p,\ \ \forall p>1.
    $$
    \par On the one hand, while \cite{zhang2007lower} assumes the boundary $\partial\Omega$ to be mean convex, our Theorem 1.1 instead assumes that the ambient manifold $(M,g)$ is non-compact.
  On the other hand, if we specialize Theorem \ref{thm:1.1} to the case $p=2,$ we obtain the inequality
    $$
    \lambda_{1,2}(\Omega) > \br{\frac{\pi}{2d}}^2,\ \ \ \lambda_{1,2}(B(x,1)) > \br{\frac{\pi}{4}}^2
    $$
    which matches the result in \cite{wang2022sharp}, except that we additionally prove the first inequality is strict.
    \par Traditionally, a powerful and widely used method for estimating the first Dirichlet eigenvalue in the linear case $p=2$ is the gradient estimate for eigenfunctions introduced by Li and Yau \cite{li1980estimates}, and further developed in \cite{1984On}. These methods extend in principle to general  $p>1.$
    \par In \cite{wang2022sharp},  the authors applied a maximum principle to a function similar to the eigenfunction to obtain a $C^0$ estimate for the Dirichlet eigenfunction when $p=2.$ Inspired of their approach, we apply the Barta's trick of $p-$Laplacian to the generalized trigonometric functions to generalize their result to the case $p>1.$

    \section{The sharp estimate of Dirichlet eigenvalue}
    \subsection{Preliminaries}

    Theorem 1.1 from \cite{jin2024lower} provides a convenient tool to derive lower bounds for $\lambda_{1,p}(\Omega)$. We state a special case of this result ($p_1=p_2=p$, $C=p-1$) which will be used in the sequel. This is also equivalent to the well-known Barta’s inequality for $p-$Laplacian. See Theorem 2.1 in \cite{allegretto1998picone}.

    \begin{lemma}[\cite{jin2024lower}]\label{lemma:2.1}
        Let $(M,g)$ be a complete Riemannian manifold and $\Omega \subseteq M$ be a bounded domain. If there exists a function $f \in W^{1,p}_\loc(\Omega)$ and a constant $\mu > 0$ such that
        \begin{equation}\label{eq:2.1}
            \Delta_p f - (p-1)|\nabla f|^p \geq \mu
        \end{equation}
        in the weak $W^{1,p}_{\rm loc}$ sense, then
        \begin{equation}\label{eq:2.2}
            \lambda_{1,p}(\Omega) \geq\mu.
        \end{equation}
        Moreover, if the $\lambda_{1,p}(\Omega)=\mu,$ then $e^{-f}$ is a $p$-eigenfunction corresponding to $\lambda_{1,p}(\Omega).$
    \end{lemma}

    \begin{proof}
        Since \cite{jin2024lower} does not prove the equality case, we briefly recall the argument for completeness.
        For any $v \in C_0^\infty(\Omega)$, we have
        \begin{align}
            \mu \int_\Omega |v|^p \,\d\upsilon_g &\leq  -p\int_\Omega |v|^{p-1} g(\nabla |v|, |\nabla f|^{p-2}\nabla f)\,\d\upsilon_g - (p-1) \int_\Omega |v|^p   |\nabla f|^p \,\d\upsilon_g\nonumber \\
            &\leq p \int_\Omega |v|^{p-1}  |\nabla v| \cdot |\nabla f|^{p-1}\,\d\upsilon_g - (p-1) \int_\Omega |v|^p   |\nabla f|^p\,\d\upsilon_g \label{eq:ineq-cauchy} \\
            &\leq p \int_\Omega \br{\frac{(|v|^{p-1}|\nabla f|^{p-1})^q}{q} + \frac{|\nabla v|^p}{p}}\,\d\upsilon_g - (p-1)\int_\Omega |v|^p   |\nabla f|^p \,\d\upsilon_g \label{eq:ineq-young} \\
            &= \int_\Omega |\nabla v|^p\,\d\upsilon_g. \nonumber
        \end{align}
        where $1/p+1/q=1$, and the last inequality is due to the Young's inequality. Notice that $W^{1,p}_0(\Omega)$
        is the completion of $C^\infty_0(\Omega)$ with respect to the Sobolev norm
        $$
        \|v\|_{1,p}=\Big(\int_\Omega (|v|^p+|\nabla v|^p)\d\upsilon_g\Big)^{\frac 1p}
        $$
         then
        the above inequality also holds for all $v\in  W^{1,p}_0(\Omega).$ Thus we obtain
        \[
            \lambda_{1,p}(\Omega) \geq \mu.
        \]
        by \eqref{eq:1.1}.
        If $\lambda_{1,p}(\Omega)=\mu,$ then the equality in \eqref{eq:2.1}, \eqref{eq:ineq-cauchy} and \eqref{eq:ineq-young} would be achieved when $v$ is a $p$-eigenfunction of $\lambda_{1,p}(\Omega).$ As a consequence,
        \[
            \nabla f \ \text{is parallel to}\ \nabla v, \quad |v|^p|\nabla f|^p = |\nabla v|^p.
        \]
        almost everywhere in $\Omega.$ Therefore
        \[
            \nabla f = -C_1 \nabla \ln v.
        \]
        for some constant $C_1.$
        This implies $f = -C_1 \ln v - C_2$, and thus $e^{-f}=C_1 e^{C_2} v$ is also a $p$-eigenfunction.
    \end{proof}

    Before we start to prove the main theorem, let us introduce the generalized trigonometric functions. For the following conclusion about generalized trigonometric functions, one may refer to \cite{dosly2005half}. Suppose $p>1,$ let
    \[
        \arcsin_p(x) := \int_0^x \frac{1}{(1-t^p)^{\frac{1}{p}}} \,\d t, \quad x \in [0,1],
    \]
    and
    \[\pi_p := 2\arcsin_p(1) = \frac{2\pi}{p \sin \frac{\pi}{p}}.\]
    Define $\sin_p(t)$ to be the inverse function of $\arcsin_p$ for $t\in [0,\frac{\pi_p}{2}]$ and set
    $$\sin_p(t)=\sin_p(\pi_p-t),\ \ t\in[\frac{\pi_p}{2},\pi_p]
    $$
    Then we extend $\sin_p(t)$ to be the odd $2\pi_p$ period function defined on the whole line $\mathbb{R}.$ It is easy to check that
    $\sin_p(t)\in C^1(\R)$ as $(\sin_p)'(\frac{\pi_p}{2})=0$ and $(\sin_p)'(\pi_p)=-1.$
    For $p=2,\ \sin_p(t)=\sin t$ and $\pi_2=\pi.$ For $p\neq 2,$ the function is not sinusoidal, but maintains similar qualitative oscillatory behavior. Furthermore, let $\cos_p(t) := (\sin_p)'(t),$ then we have the following generalized Pythagorian identity:
    \begin{equation}\label{eq:sinp}
        |\sin_p(t)|^p+|\cos_p(t)|^p = 1,\ \ \forall t\in \mathbb{R}
    \end{equation}
    We also introduce the cotangent functions $\cot_p(t)=\frac{\cos_p(t)}{\sin_p(t)}$ which satisfies that
    $$
    (\cot_p)'(t)=-|\cot_p(t)|^{2-p}-|\cot_p(t)|^{2}
    $$
    according to the generalized Pythagorian identity.
    \begin{lemma}\label{lem:2.2}
        The Euler Beta function
        \[
            \mathrm{B}(\alpha,\beta) = \int_0^1 x^{\alpha-1}(1-x)^{\beta-1}\,\d x
        \]
        can be represented as
        \[
            \mathrm{B}(\alpha,\beta) = p \int_0^{\frac{\pi_p}{2}} \big(\sin_p(t)\big)^{p \alpha - 1} \big(\cos_p(t)\big)^{p\beta-p+1} \,\d t.
        \]
    \end{lemma}

    \begin{proof}
        By changing the variable $x = \big(\sin_p(t)\big)^p$, the lemma is easily proved.
    \end{proof}
\subsection{Proof of the main results}
    \begin{proof}[Proof of Theorem \ref{thm:1.1}.]
        Let $o\in\Omega$ be a fixed point and there exists a ray $$\gamma \colon [0,+\infty) \to M$$ starting from $o$ as $M$ is non-compact.
      We consider the Busemann function with respect to $\gamma$,
    \[
        \beta_\gamma \colon M \to \R, \quad \beta_\gamma(z) := \lim_{t \to \infty} (t - d(z,\gamma(t))).
    \]
    For the simplicity of notations, we denote $\beta = \beta_\gamma$. By Busemann's book \cite{busemann1955}, we have $|\nabla \beta| = 1$ almost everywhere, therefore
    $$
    |\beta(x)-\beta(y)|\leq  d,\ \ \forall x,y\in\Omega,
    $$
    where $d$ is the diameter of $\Omega$. So there exists a constant $a\in\mathbb{R},$ such that
    $$
    \beta(x)\in [a,a+d],\ \ \forall x\in\overline{\Omega}.
    $$
    The Busemann function $\beta$ cannot be a constant function in $\Omega$ as $|\nabla \beta|=1$ almost everywhere. Furthermore, since $\mathrm{Ric} \geq 0$, we have that $\Delta \beta \geq 0$ in the sense of distribution by the Laplace Comparison Theorem. Then $\beta$ cannot attain its maximum in the interior of $\Omega.$ Hence
    $$
    \beta(x)\in [a,a+d),\ \ \forall x \in\Omega.
    $$
       \par Now let us define the test function
        $$f(\cdot) = -\ln \sin_p \br{A\big(a+d-\beta(\cdot)\big)}
        \ \ {\rm with}\ A=\frac{\pi_p}{2d}$$
        on $\Omega\subseteq\beta^{-1}[a,a+d).$
         By the properties of the $p-$sine function we obtain that
        \begin{align}
            \dot f(\beta) &= A \cot_p\br{A(a+d-\beta)} > 0, \nonumber\\
            \ddot f(\beta) &= A^2 \Big(\cot_p\big(A(a+d-\beta)\big)\Big)^{2-p} + A^2\Big(\cot_p\big(A(a+d-\beta)\big)\Big)^2 \nonumber \\
            &= A^p \dot f(\beta)^{2-p} + \dot f(\beta)^2. \label{eq:2.5}
        \end{align}
        Recall that $\Delta \beta \geq 0$ in distributional sense and together with \eqref{eq:2.5} and that $|\nabla \beta| = 1$ almost everywhere \cite{busemann1955}, we obtain that for all $v \in C^\infty_0(\Omega)$,
        \begin{align}
            \int_\Omega v \Delta_p f &= \int_\Omega v\Div \br{|\nabla f|^{p-2}\nabla f} \nonumber \\
            &= \int_\Omega v \Div (\dot f(\beta)^{p-1}|\nabla \beta|^{p-2} \nabla \beta) \nonumber \\
            &= \int_\Omega v \Div(\dot f(\beta)^{p-1}\nabla \beta) \nonumber \\
            &= -\int_\Omega \dot f(\beta)^{p-1} g(\nabla \beta, \nabla v) \nonumber  \\
            &= \int_\Omega \br{\dot f(\beta)^{p-1} \Delta \beta + (p-1) \dot f(\beta)^{p-2} \ddot f(\beta) |\nabla \beta|^2} v \nonumber \\
            &\geq \int_\Omega (p-1)\ddot f(\beta)\dot f(\beta)^{p-2} v \label{eq:lap-b-geq-0} \\
            &= \int_\Omega (p-1)\br{A^p + \dot f(\beta)^{p}} v. \nonumber
        \end{align}
        Therefore,
        \[
            \Delta_p f- (p-1)|\nabla f|^p \geq (p-1)A^p
        \]
       in the sense of distributional. Applying Lemma \ref{lemma:2.1}, we derive that
        \[
            \lambda_{1,p}(\Omega) \geq (p-1)A^p = (p-1)\br{\frac{\pi_p}{2d}}^p.
        \]

        We now show this inequality is strict. If the equality holds, then by the last statement in Lemma \ref{lemma:2.1}, $e^{-f} = \sin_p (A(a+d-\beta))$ must be a $p$-eigenfunction. This would imply that $\beta|_{\partial \Omega} \equiv a + d$. Furthermore, $\beta$ is harmonic in $\Omega$ due to the equality \eqref{eq:lap-b-geq-0}. Hence $\beta$ would have to be a constant in $\Omega,$
          contradicting $|\nabla\beta|\equiv 1$ on $\gamma\cap\Omega.$ Therefore, the inequality is strict.
        \\
        \par To prove the estimate is sharp, let $\eps > 0$ be a small positive number and we consider a smooth bump function
        \[
            j_\eps(r) = \left\{\begin{array}{ll}
                e^{\frac{1}{\eps^2}}e^{\frac{1}{r^2-\eps^2}}, & \quad 0 \leq r < \eps, \\
                0, & \quad r \geq \eps.
            \end{array}\right.
        \]
        Define
        \[
            f_\eps(r) = \int_0^r j(t) \,\d t,\ \ r\geq 0.
        \]
         and construct the warped product manifold
         $$(\R^n, g_\eps=  \d r^2 + f_\eps(r)^2 \d \mathbb{S}^{n-1}).$$
         This metric is smooth as $\dot f_\eps(0)=1$ and $f^{({\rm even})}_\eps(0)=0.$ Furthermore, since  $\ddot f_\eps \leq 0,$ we have that $\operatorname{Ric} g_\eps \geq 0$.
          \par Let $o$ be the origin of $\R^n$ and take $z_\eps = \exp_o((1-\delta)\theta)$ where $\theta \in \mathbb{S}^{n-1}$ and $\delta = f_\eps(\eps) < \eps$. We claim that
        \[
            B(o,2-2 \pi \delta-\delta) \subseteq B(z_\eps,1).
        \]
        For any $x \in B(o,2-2 \pi \delta-\delta)$, we write $x = \exp_o (r_x \theta_x)$ for some $r_x \leq 2 - 2 \pi \delta - \delta$ and $\theta_x \in \mathbb{S}^{n-1}$. Let $y = \exp_o((1-\delta)\theta_x)$, we have
        \[
            d(x,z_\eps) \leq d(x,y) + d(y,z_\eps) \leq 1 - 2\pi\delta + \pi \delta < 1.
        \]
        Thus the claim is proved.
        \par Denote $r = 2 - 2 \pi \delta - \delta$ and $\tau = \frac{\pi_p}{2r}.$ Define the function
        $$\fai(x) = \sin_p \big(\tau(r-d(x,o))\big) \in W^{1,p}_0(B(o,r)).
        $$ Then by \eqref{eq:1.1},
        \begin{align*}
            \lambda_{1,p}(B(z_\eps,1)) &\leq \lambda_{1,p}(B(o,2-2 \pi \delta-\delta))
            \\ &
             \leq \frac{\int_{B(o,r)}|\nabla \fai|^p \, \d\upsilon_{g_\eps}}{\int_{B(o,r)}|\fai|^p \, \d\upsilon_{g_\eps}} \\
            &= \frac{\tau^p\int_0^r \Big(\cos_p(\tau(r-s))\Big)^p \omega_{n-1}f_\eps(s)^{n-1} \, \d s}{\int_0^r \Big(\sin_p(\tau(r-s))\Big)^p \omega_{n-1}f_\eps(s)^{n-1} \, \d s} \\
             &\leq \frac{ \tau^p\int_0^r \Big(\cos_p(\tau(r-s))\Big)^p \omega_{n-1}f_\eps(\eps)^{n-1} \, \d s}{\int_\eps^r \Big(\sin_p(\tau(r-s))\Big)^p \omega_{n-1}f_\eps(\eps)^{n-1} \, \d s} \\
            &\leq \frac{\tau^p\int_0^{\pi_p/2} \big(\cos_p(t) \big)^p\, \d t}{\int_{0}^{\tau(r-\eps)} \big(\sin_p(t)\big)^p \, \d t}.
        \end{align*}
        Let $\eps \to 0$ and according to  Lemma \ref{lem:2.2}, we obtain that
        \begin{align*}
           \limsup_{\eps \to 0} \lambda_{1,p}(B(z_\eps,1)) &\leq \br{\frac{\pi_p}{4}}^p \frac{\int_0^{\pi_p/2} (\cos_p)^p(t) \, \d t}{\int_{0}^{\pi_p/2} (\sin_p)^p(t) \, \d t}
           \\ & =\br{\frac{\pi_p}{4}}^p \frac{\mathrm{B}\br{\frac{1}{p},\frac{2p-1}{p}}}{\mathrm{B}\br{\frac{p+1}{p},\frac{p-1}{p}}}
           \\ &= (p-1)\br{\frac{\pi_p}{4}}^p.
        \end{align*}
    On the other hand, due to Theorem \ref{thm:1.1}, we also have that
        $$
          \liminf_{\eps \to 0} \lambda_{1,p}(B(z_\eps,1)) \geq \liminf_{\eps \to 0} (p-1)\br{\frac{\pi_p}{2d_\eps}}^p=(p-1)\br{\frac{\pi_p}{4}}^p
        $$
        where $d_\eps$ denotes the diameter of $B(z_\eps,1)$. Notice that $\exp_o((2-\delta)\theta) \in B_1(z_\eps)$ as $z_\eps = \exp_o((1-\delta)\theta)$. Therefore
        \[
            2-\delta \leq d(B_1(z_\eps)) \leq 2,
        \]
        which shows that $d(B_1(z_\eps)) \to 2$ as $\eps\rightarrow 0.$
        In the end, we obtain that
        $$
         \lim_{\eps \to 0} \lambda_{1,p}(B(z_\eps,1)) =(p-1)\br{\frac{\pi_p}{4}}^p,\ {\rm with} \ \lim_{\eps \to 0}d(B_1(z_\eps)) \to 2.
        $$
        Thus we get the sharpness Theorem \ref{thm:1.1} as long as we consider a sequence of Riemannian manifolds and bounded domains:
        $$
        (\R^n,\bar{g}_\eps)=\Big(\R^n,\big(\frac{d}{d_\eps}\big)^2g_\eps\Big),\ \ \Omega_\eps=B_{g_\eps}(z_\eps,1).
        $$
        Then
        $$ {\rm diam}( \Omega_\eps,\bar{g}_\eps)=\frac{d}{d_\eps}{\rm diam}( \Omega_\eps,g_\eps)=d
        $$
        and
        $$
        \lambda_{1,p}(\Omega_\eps,\bar{g}_\eps)=\Big(\frac{d}{d_\eps}\Big)^{-p} \lambda_{1,p}(\Omega_\eps,g_\eps)\rightarrow (p-1)\br{\frac{\pi_p}{2d}}^p,\ {\rm as}\ \eps\rightarrow 0.
        $$
    \end{proof}

    \begin{remark}
    \begin{itemize}
      \item In \cite{wang2022sharp}, the authors proved the sharpness of the lower bound in the case $p=2$ by constructing a sequence of manifolds to approximate the bound. We point out that there is a trivial flaw in their construction as $g_i$'s in the proof of Theorem 2.3 in \cite{wang2022sharp} are not smooth at $o.$ Here we fixed this detail.
      \item The condition that $M$ is non-compact cannot be dropped. Let us consider a sequence of manifolds $M_\eps = ([0,\pi-\eps] \times \mathbb{S}^{n-1}, \d r^2 + \sin^2 r g_{\mathbb{S}^{n-1}})$ embedded in $S^n$ with small $\eps$ and the diameter of $M_\eps$ is 2.
            Assume $p<n$ and recall that
    \[
      \lambda_{1,p}(M_\eps) = \inf_{v \in W^{1,p}_0(M_\eps)\setminus\{0\}} \frac{ \int_{M_\eps} |\nabla v|^p \d v_g}{ \int_{M_\eps} |v|^p \d v_g}.
    \]
    For each $\eps$, we define a function $f_\eps(r) \in W^{1,p}_0(M_\eps)$ on $M_\eps$ which depends only on the distance $r$ of $o \in [0,\pi-\eps] \times \mathbb{S}^{n-1}$ as follows
    \[
      f_\eps(r) = \left\{\begin{array}{ll}
        1, & \quad r \in [0,\pi-2\eps], \\
        \frac{\pi-\eps-r}{\eps}, & \quad r \in [\pi-2\eps,\pi-\eps].
      \end{array}\right.
    \]
    Then
    \begin{align*}
      \lambda_{1,p}(M_\eps) &\leq \frac{ \int_{M_\eps} |\nabla (f_\eps(r))|^p \d v_g}{ \int_{M_\eps} f_\eps(r)^p \d v_g} \\
      &= \frac{\int_0^{\pi-\eps} \dot f_\eps(r)^p \sin^{n-1} r \,\d r}{\int_0^{\pi-\eps} f_\eps(r)^p \sin^{n-1} r \,\d r} \\
      &\leq C \int_{\pi-2\eps}^{\pi-\eps} \eps^{-p} \sin^{n-1} r \,\d r \\
      &\leq C \int_\eps^{2\eps} \eps^{-p} r^{n-1} \,\d r \\
      &\leq C \frac{2^n-1}{n}\eps^{n-p},
    \end{align*}
    where $C = \br{\int_0^\pi \sin^{n-1} r \,\d r}^{-1} > 0$. Therefore $\lambda_{1,p}(M_\eps) \to 0$ as $\eps \to 0$.

    Now assume $p=n$, we define
    \[
      f_\eps(r) = \left\{\begin{array}{ll}
        1, & \quad r \in [0,\pi/2], \\
        \frac{\ln(\pi-r)-\ln \eps}{\ln(\pi/2)-\ln \eps}, & \quad r \in [\pi/2,\pi-\eps].
      \end{array}\right.
    \]
    Therefore
    \begin{align*}
      \lambda_{1,n}(M_\eps) &\leq \frac{\int_0^{\pi-\eps} \dot f_\eps(r)^n \sin^{n-1} r \,\d r}{\int_0^{\pi-\eps} f_\eps(r)^n \sin^{n-1} r \,\d r} \\
      &\leq C \int_{\pi/2}^{\pi-\eps} (\ln(\pi/2)-\ln \eps)^{-n} \br{\frac{1}{\pi-r}}^n \sin^{n-1} r\,\d r \\
      &= C \frac{\int_{\eps}^{\pi/2} \frac{\sin^{n-1}r}{r^n} \,\d r}{(\ln(\pi/2)-\ln \eps)^{n}}\\
      &\rightarrow 0 \ \text{as}\ \eps\rightarrow0.
    \end{align*}
    Here $C>0$ is a lower bound for $\int_0^{\pi-\eps} f_\eps(r)^n \sin^{n-1} r \,\d r$ which is independent of $\eps$.  Recall that Theorem \ref{thm:1.1} gives a uniform positive lower bound for $\lambda_{1,p}(M_\eps)$, then Corollary \ref{coro:1.3} is proved provided we take $D=M_\eps.$

      \item Recall the relationship between  $\lambda_{1,p}$ and the Maz'ya constant of $\Omega$ in  \cite{maz2003lectures, grigor1999isoperimetric}:
      $$
      \lambda_{1,p}(\Omega)\leq m_p(\Omega):=\inf_{K\subseteq\Omega}
      \frac{{\rm cap}_p(K,\Omega)}{{\rm vol}(K)}
      $$
      where ${\rm cap}_p(K,\Omega)$ is the $p-$capacity of $(K,\Omega).$ Then under the conditions of Theorem \ref{thm:1.1}, and for any compact smooth domain $K$ in $\Omega,$
      $$
      (p-1)\Big(\frac{\pi_p}{2{\rm diam}(\Omega)}\Big)^p {\rm vol}(K) < {\rm cap}_p(K,\Omega)=\inf\Big\{\int_{\Omega}|\nabla f|^p\,\d\upsilon_g:\ C_0^1(\Omega)\ni f\ge 1_K\Big\}
      $$
      Let $p\rightarrow 1,$ and we obtain that
      $$
      \frac{{\rm vol}(K)}{{\rm diam}(\Omega)}\leq {\rm cap}_1(K,\Omega)=\inf \big\{{\rm vol}(\partial O):\ K\subseteq\text{smooth}\ O\subseteq \Omega\big\}\leq {\rm vol}(\partial K).
      $$
      Hence we have the following estimate of the Cheeger constant $h(\Omega).$
      $$
      h(\Omega):= \inf_{K\subseteq\Omega}\frac{{\rm vol}(\partial K)}{{\rm vol}(K)}\geq  \frac{1}{{\rm diam}(\Omega)}
      $$
      and we can show that the estimate is also sharp by considering the warped product manifold as above.
      One can see \cite{kawohl2003isoperimetric} for more information about $h(\Omega)$ and $\lambda_{1,p}(\Omega).$
      \item Let $p\rightarrow\infty,$ then under the conditions of Theorem \ref{thm:1.1},
    $$
      \lambda_{1,\infty}(\Omega):=\lim_{p \to\infty} \lambda_{1,p}(\Omega)^{\frac 1p} \geq \frac{1}{{\rm diam}(\Omega)}
    $$
    Here $\lambda_{1,\infty}(\Omega)$
    is related to the so-called $\infty-$Laplacian, one can see Remark 3.4 in \cite{xiao2009p} for more details.
    
    \end{itemize}

    \end{remark}

\noindent{Xiaoshang Jin}\\
  School of mathematics and statistics, Huazhong University of science and technology, Wuhan, P.R. China. 430074
 \\Email address: jinxs@hust.edu.cn
\bigskip \\
\noindent{Zhiwei L\"u}\\
  School of mathematics and statistics, Huazhong University of science and technology, Wuhan, P.R. China. 430074
 \\Email address: m202470005@hust.edu.cn

\end{document}